%% file: Accessible_groups.tex
\documentclass[a4paper]{amsart}

\usepackage[english]{babel}
\usepackage[ansinew]{inputenc}
\usepackage{amsmath, amsthm, amssymb, graphicx}
\usepackage[all]{xy}
\usepackage{color}

\newcommand{\Ab}{\mathbb A}
\newcommand{\Ac}{\mathcal A}

\newcommand{\Bb}{\mathbb B}

\newcommand{\Cb}{\mathbb C}

\newcommand{\Ec}{\mathcal E}
\newcommand{\Fb}{\mathbb F}
\newcommand{\Fc}{\mathcal F}
\newcommand{\Hc}{\mathcal H}
\newcommand{\Kc}{\mathcal K}
\newcommand{\Nb}{\mathbb N}
\newcommand{\Nc}{\mathcal N}
\newcommand{\Rc}{\mathcal R}
\newcommand{\Sc}{\mathcal S}
\newcommand{\Tc}{\mathcal T}
\newcommand{\Wc}{\mathcal K}
\newcommand{\Wh}{\mathcal W}
\newcommand{\Zb}{\mathbb Z}

\newcommand{\inv}{^{-1}}
\newcommand{\epm}{^{\pm 1}}
\newcommand{\lab}{\mathcal L}
\newcommand{\la}{\langle}
\newcommand{\ra}{\rangle}

\DeclareMathOperator{\Aut}{Aut}
\DeclareMathOperator{\Out}{Out}
\DeclareMathOperator{\Inn}{Inn}
\DeclareMathOperator{\id}{Id}
\DeclareMathOperator{\Norm}{N}
\DeclareMathOperator{\Centr}{Z}
\DeclareMathOperator{\Mod}{Mod}
\DeclareMathOperator{\image}{im}

\newtheorem{theorem}{Theorem}[section]
\newtheorem{proposition}[theorem]{Proposition}
\newtheorem{lemma}[theorem]{Lemma}

\newtheorem{corollary}[theorem]{Corollary}

\theoremstyle{remark}
\newtheorem*{remark}{Remark}
\newtheorem{definition}{Definition}

\title{The automorphism group of accessible groups}
\author{Mathieu Carette}
\address{D\'epartement de math\'ematiques \\ Universit\'e Libre de Bruxelles ULB CP 216\\ Boulevard du triomphe \\ 1050 Brussels\\ Belgium}
\email{mcarette@ulb.ac.be}
\thanks{The author is a FNRS Research Fellow (Belgium). Part of this research was supported by a Marie Curie fellowship.}

\begin{document}
	
	\begin{abstract}
		In this article, we study the outer automorphism group of a group $G$ decomposed as a finite graph of groups with finite edge groups and finitely generated vertex groups with at most one end. We show that $\Out(G)$ is essentially obtained by taking extensions of relative automorphism groups of vertex groups, groups of Dehn twists and groups of automorphisms of free products. We apply this description and obtain a criterion for $\Out(G)$ to be finitely presented, as well as a necessary and sufficient condition for $\Out(G)$ to be finite. Consequences for hyperbolic groups are discussed.
	\end{abstract}
	
	\maketitle
	
	\section{Introduction}
		
		In this article, we study the structure of the automorphism group of accessible groups. A group is called \emph{accessible} if it admits a decomposition as a finite graph of groups with finite edge groups and finitely generated vertex groups with at most one end. Recall that finitely presented groups are accessible \cite{Dunwoody}, as well as finitely generated groups with a uniform bound on the order of finite subgroups \cite{Linnell}. We reduce the study of the automorphism group of an accessible group to that of certain relative automorphism groups of vertex groups, groups of Dehn twists and automorphism groups of free products. We briefly discuss these three classes of groups.
		
		We recall the notion of relative automorphism groups. Let $G$ be a group, and let $\Hc$ be a family of subgroups of $G$, closed under taking conjugates. The relative automorphism group $\Out_\Hc(G)$ is the subgroup of $\Out(G)$ of those outer automorphisms that preserve the conjugacy class of each element of $\Hc$. This is an algebraic analogue of the following geometric situation: $S$ is an orientable surface with boundary, and $\Hc$ is the family of cyclic subgroups of $G=\pi_1(S)$ corresponding to boundary curves, then $\Out_\Hc(G)$ is exactly the mapping class group $\Mod^\partial(S)$ preserving the boundary componentwise.
		
		A Dehn twist is a natural generalization to any splitting of the corresponding notion in surfaces. The general definition of a Dehn twist is deferred to the next section, see Definition~\ref{defn:Dehn_twist}. Groups generated by Dehn twists of a given splitting $\Ab$ occur naturally when studying the group $\Aut^\Ab(G)$ of automorphisms of a group $G$ preserving a graph of group decomposition $\Ab$. The structure of such groups of automorphisms has been well studied, see for example \cite{Levitt}.
		
		Suppose $G$ is a free product of finitely many freely indecomposable groups. Fouxe-Rabinovitch \cite{FR1,FR2} gave a presentation of $\Aut(G)$ in terms of the freely indecomposable free factors of $G$ and their automorphism group. In particular he shows that if each freely indecomposable free factor of $G$ is finitely presented and has a finitely presented automorphism group, then $\Aut(G)$ is finitely presented. Gilbert \cite{Gilb} obtained the same result using peak-reduction methods, which we adapt to our setting. 
		
		Our study of the structure of the automorphism group of an accessible group $G$ combining the three classes of groups above yield the following consequence for the finite presentability of $\Out(G)$.
	\begin{theorem} \label{theorem:out_fp}
		Let $\Ab$ be a finite reduced graph of groups with finite edge groups and finitely generated vertex groups with at most one end. Let $G=\pi_1(\Ab)$ act on the Bass-Serre tree $T_A$.
		Suppose the following two conditions are satisfied:
		\begin{enumerate}
			\item \label{norm_fp_center_fg} For each edge $e$ of $T_A$ and each vertex $v$ stabilized by $G_e$, the normalizer $N_e$ of $G_e$ in $G_v$ is finitely presented, and its center $\Centr(N_e)$ is finitely generated.
			\item For each vertex $v$ of $T_A$, the group $\Out_{\Hc_v}(G_v)$ of automorphisms relative to the family $\Hc_v$ of edge stabilizers contained in $G_v$ is finitely presented. \label{vertex_gp_fp}
		\end{enumerate}
		Then $\Out(G)$ is finitely presented.
	\end{theorem}

	Our study also leads to a characterization of accessible groups with finite outer automorphism group in terms of splittings over finite groups and relative automorphism groups of maximal elliptic subgroups. 
		\begin{theorem} \label{theorem:finite_out}
			Let $\Ab$ be a finite reduced graph of groups with finite edge groups and finitely generated vertex groups with at most one end. Let $G = \pi_1(\Ab)$ act on the Bass-Serre tree $T_A$. Then $\Out(G)$ is infinite if and only if one of the following holds:
			\begin{enumerate}
				\item There is a vertex stabilizer $G_v$ of $T_A$ such that $\Out_{\Hc_v}(G_v)$ is infinite, where $\Hc_v$ is the family of edge stabilizers contained in $G_v$;
				\item There is a splitting of $G$ as an amalgam $A \ast_C B$ over a finite group with $B \neq C$ such that the center of $A$ has infinite index in the centralizer of $C$ in $A$; 
				\item There is a splitting of $G$ as an HNN extension $A \ast_C$ over a finite group such that the centralizer of $\tilde C$ in $A$ is infinite, where $\tilde C$ is one of the two isomorphic copies of $C$ in $A$ given by the HNN extension. 
			\end{enumerate}
		\end{theorem}
		A well-known theorem of Paulin \cite{Paulin}, combined with Rips' theory \cite{BestFeig}, implies that if $G$ is a one-ended hyperbolic group with $\Out(G)$ infinite, then $G$ splits over a $2$-ended subgroup. In fact Levitt \cite{Levitt} showed that a one-ended hyperbolic group $G$ has infinite outer automorphism group if and only if $G$ splits over a $2$-ended subgroup with infinite center, either as an arbitrary HNN-extension, or as an amalgam of groups with finite centers. Theorem~\ref{theorem:finite_out} allows us to drop the condition that $G$ is one-ended, yielding a characterization applying to all hyperbolic groups in Theorem~\ref{theorem:hyp_gp_infinite_out}. A refined version of Theorem~\ref{theorem:finite_out} has been obtained independently by Guirardel and Levitt \cite{GuirLev} for relatively hyperbolic groups. In particular, they recover Theorem~\ref{theorem:hyp_gp_infinite_out}.
		
		The paper is organized as follows. We introduce notation and make some preliminary observations in Section~\ref{sec:prelim}. The structure of the automorphism group of an accessible group is described in Section~\ref{sec:structure_out}, summarized by Proposition~\ref{proposition:main_tech} which is the main technical result of this paper. We apply our structural results inductively in Section~\ref{sec:finite_pres} to prove Theorem~\ref{theorem:out_fp}. Section~\ref{sec:finite_out} is devoted to the proof of Theorem~\ref{theorem:finite_out}. In Section~\ref{sec:hyperbolic_gps}, Theorems~\ref{theorem:out_fp} and~\ref{theorem:finite_out} are applied to show that the automorphism group of any hyperbolic group is finitely presented and to characterize hyperbolic groups with finite outer automorphism group.
	
	\section{Preliminaries} \label{sec:prelim}
		
		We denote the center of a group $G$ by $\Centr(G)$. If $H$ is a subgroup of $G$, we write $\Centr_G H$ for the centralizer of $H$ in $G$ and $\Norm_G H$ for the normalizer of $H$ in $G$. For an element $g \in G$, we define the associated inner automorphism as $i_g : G \to G : x \mapsto gxg\inv$.
		
		\begin{definition}\label{defn:Dehn_twist}
			Let $G$ be a group. An automorphism $\psi \in Aut(G)$ is called a \emph{Dehn twist} if one of the following holds.
			\begin{enumerate}
				\item $G$ splits as an amalgam $A \ast_C B$, and there is an element $z \in \Centr_A(C)$ such that $\psi|_A = \id|_A$ and such that $\psi|_B = i_z|_B$.
				\item $G$ splits as an HNN-extension $A \ast_C=\langle A,t| \Rc_A, t\omega(c)t\inv = \alpha(c) \text{ for } c \in C\rangle$, and there is an element $z \in \Centr_A(\alpha(C))$ such that $\psi|_A = \id|_A$ and $\psi(t) = zt$.
			\end{enumerate}
		\end{definition}%
		%
		\subsection{Graphs of groups}
			In order to fix notation, we recall some definitions and results of Bass-Serre theory. The unfamiliar reader is referred to \cite{Serre} for more details.
			
			A \emph{graph} $A$ is given by the following data: a set of vertices $VA$, a set of edges $EA$, a boundary map $\alpha:EA \to VA$ and an involution \mbox{$\inv : EA \to EA$} such that $e\inv \neq e$ for each edge $e \in EA$. The vertex $\alpha(e)$ is called the \emph{initial vertex} of $e$. The \emph{terminal vertex} of $e$ is defined as $\omega(e) := \alpha(e\inv)$.
			
			A \emph{graph of groups} $\Ab$ is given by the following data: a connected graph $A$, a vertex group $\Ab_v$ for each vertex $v \in VA$, an edge group $\Ab_e$ for each edge $e \in EA$ such that $\Ab_e = \Ab_{e\inv}$ and injections $\alpha_e:\Ab_e \to \Ab_{\alpha(e)}$ of each edge group in the initial vertex group. Given such data, we also define the map $\omega_e :\Ab_e \to \Ab_{\omega(e)}$ by $\omega_e := \alpha(e\inv)$. An $\Ab$-path is a sequence $a_0,e_1,a_1,...,e_n,a_n$ where $a_i$ is an element of a vertex group $\Ab_{v_i}$ and $e_i$ is an edge of $\Cb$ such that $\omega(e_i) = v_i = \alpha(e_{i+1})$. Two $\Ab$-paths $\gamma_1$ and $\gamma_2$ are \emph{elementarily equivalent} if either
			\[ \gamma_1 = \gamma,a_i,e,1,e\inv,a_{i+2},\gamma' \text{ and }  \gamma_2 = \gamma,a_ia_{i+2},\gamma'\]
			or if
			\[ \gamma_1 = \gamma,a_i,e,a_{i+1},\gamma' \text{ and }  \gamma_2 = \gamma,a_i\alpha_e(c),e,\omega_e(c\inv)a_{i+1},\gamma' \text{ where } c \in \Ab_e \]
			Let $\sim$ denote the equivalence relation on the set of $\Ab$-paths generated by this elementary equivalence. If $u_0$ is a vertex of $A$, then the \emph{fundamental group} $\pi_1(\Ab,u_0) = \{\text{closed } \Ab\text{-paths based at } u_0 \}/\sim$ is a group with the operation of concatenation. The equivalence class of an $\Ab$-path $\gamma$ is denoted by $[\gamma]$. If $\Ab$ is a graph of groups, we use the same letter $A$ to denote the underlying graph and we let $T_A$ be the Bass-Serre tree on which the fundamental group $\pi_1(\Ab,u_0)$ acts. The isomorphism class of $\pi_1(\Ab,u_0)$ does not depend on the basepoint $u_0$, so we will often write $\pi_1(\Ab)$.
			
			All actions on trees are assumed not to invert edges. We say a group $G$ \emph{splits} over a subgroup $C$ if either $G = A \ast_C B$ with $A \neq C \neq B$ or if $G=A \ast_C$.
			
			Let $\Ab$ be a finite graph of groups. A subgroup $H$ of $\pi_1(\Ab)$ is \emph{elliptic} if it fixes a vertex in $T_A$. For a vertex or edge $x$ of $T_A$ we let $G_x$ be the stabilizer of $x$ in $G$. $\Ab$ is \emph{minimal} if $\Ab$ does not have any vertex $v$ of valence 1 such that the boundary monomorphism of the adjacent edge is surjective. $\Ab$ is \emph{reduced} if for any non-loop edge $e \in EA$, the boundary monomorphism $\alpha_e$ is not surjective. Note that a reduced graph of groups is automatically minimal.
			
		\subsection{Normalizers of elliptic subgroups} 
			
			\begin{lemma} \label{lemma:normvertexgp}
				Let $\Ab$ be a graph of groups, and $G_v$ be a vertex stabilizer in $G=\pi_1(\Ab)$ that does not stabilize an edge of $T_A$. Then $G_v$ is its own normalizer in $G$.
			\end{lemma}
			\begin{proof}
				The normalizer $\Norm_G G_v$ acts on the set of fixed points of $G_v$ in $T_A$, which only consists of the vertex $v$, thus showing that $\Norm_G G_v \subset G_v$.
			\end{proof}
			
			The following lemma computes the normalizer of an elliptic subgroup of $G=\pi_1(\Ab)$ from normalizers in vertex groups.
			\begin{lemma} \label{lemma:norm}
				Let $\Ab$ be a finite graph of groups with finite edge groups. Let $H$ be an elliptic subgroup of $G=\pi_1(\Ab)$.
				\begin{enumerate}
					\item If $\Norm_{G_v}H$ is finitely generated (resp. presented) for every vertex $v \in T_A$ fixed by $H$  then $\Norm_G H$ is finitely generated (resp. presented).
					\item If the center of $\Norm_{G_v}H$ is finitely generated for each $v\in T_A$ fixed by $H$, then the center of $\Norm_G H$ is finitely generated.
				\end{enumerate}
			\end{lemma}
			\begin{proof}
				Let $T_H$ be the maximal subtree of $T_A$ fixed by $H$. Since $H$ is elliptic, $T_H$ is nonempty. The normalizer $\Norm_G H$ acts on $T_H$, and let $\Bb$ be the corresponding graph of groups. If $v$ is a vertex of $T_H$, then the stabilizer of $v$ in $\Norm_G H$ is $\Norm_G H \cap G_v = \Norm_{G_v} H$. Let us show that $\Bb$ has finitely many edges. The inclusions $T_H \hookrightarrow T_A$ and $\Norm_G H \hookrightarrow G$ induce a graph map: $B \to A$. We claim that that for each edge $e \in EA$, there are only finitely many preimages in $EB$. Let $f_1,f_2$ be two edges of $T_H$ in the same $G$-orbit, and let $X$ be the set of elements of $G$ sending $f_2$ to $f_1$. Then $f_1$ and $f_2$ are in the same $\Norm_G H$ orbit if and only if $H^X = H^{G_{f_1}}$. Therefore there cannot be more preimages of $e$ in $EB$ than there are conjugacy classes of groups isomorphic to $H$ in the finite group $G_e$. Hence $\Bb$ is a finite graph of groups such that each vertex group is isomorphic to some $\Norm_{G_v}H$ and each edge group is finite. Thus $\pi_1(\Bb) = \Norm_G H$ is finitely generated (resp. presented) provided that each $\Norm_{G_v}H$ is.
				
				We now consider the center of $\Norm_G H=\pi_1(\Bb)$. Note that if $\Bb'$ is obtained from $\Bb$ by collapsing some non-loop edges with at least one surjective boundary monomorphism, then the set of vertex stabilizer for $\Bb'$ is a subset of the set of vertex stabilizer for $\Bb$. Thus without loss of generality, we can suppose that $\Bb$ is reduced. If $\Bb$ consists of a single vertex, the center of $\Norm_G H$ is the same as the center of $\Norm_{G_v} H$ for some vertex $v \in VT_A$. If $\Bb$ is a mapping torus, then the center of $\Norm_G H$ is virtually infinite cyclic. In any other case, the center of $\Norm_G H$ is contained in all edge stabilizers of $T_B$, and so must be finite.
			\end{proof}
			
		\subsection{Relative automorphism groups}
			
			The following lemma is well-known, but we include a proof for the sake of completeness.
			\begin{lemma} \label{lemma:edge_stab_canonical}
				Let $\Ab$ and $\Ab'$ be finite reduced graphs of groups with finite edge groups and finitely generated vertex groups with at most one end. Let $\Hc$ and $\Hc'$ be the families of edge stabilizers of $G=\pi_1(\Ab)$ and $G'=\pi_1(\Ab')$ respectively. If $\varphi : G \to G'$ is an isomorphism, then $\varphi$ maps $\Hc$ to $\Hc'$. In particular $\Out_\Hc(G)$ has finite index in $\Out(G)$. 
			\end{lemma}
			\begin{proof} We prove the first assertion by induction on the number of edges of $\Ab$. If $\Ab$ has no edge, then $\Hc$ is empty and $G$ has at most one end. Hence $G'$ also has at most one end, so $\Ab'$ has no edge and $\Hc'$ is empty.
			
				Suppose now that $|EA|>0$. Let $\Fc$ be the family of finite subgroups of $G$ over which $G$ splits, and let $\Fc_{\min}$ be the minimal elements of this family with respect to inclusion. Defining $\Fc'$ and $\Fc'_{\min}$ similarly, it is clear that $\varphi$ maps $\Fc$ to $\Fc'$ and $\Fc_{\min}$ to $\Fc'_{\min}$.
			
				Let $\Bb$ be the graph of groups obtained by collapsing all edges of $\Ab$ whose edge stabilizer are not in $\Fc_{\min}$. As edge stabilizers of $T_A$ are finite, any finite group over which $G$ splits contains a minimal subgroup over which $G$ splits, so that $\Bb$ is not a single vertex and $\Fc_{\min}$ is nonempty. Let $\Sc$ be the family of vertex stabilizers of $T_B$. 
				
				Note that if a subgroup of $G$ does not split over a subgroup of an element of $\Fc_{\min}$ then it acts elliptically on $T_B$. Therefore, $\Sc$ is characterized as the family of maximal subgroups of $G$ which do not split over a subgroup of an element of $\Fc_{\min}$. Again, defining $\Bb'$ and and $\Sc'$ similarly, it is clear that $\varphi$ maps $\Sc$ to $\Sc'$. Therefore, $\varphi$ maps vertex stabilizers of $T_B$ to vertex stabilizers of $T_{B'}$. Since $\Ab$ and $\Ab'$ are reduced, so are $\Bb$ and $\Bb'$. Hence vertex groups of $\Bb$ (resp. $\Bb'$) correspond bijectively to conjugacy classes of elements of $\Sc$ (resp. $\Sc'$). Therefore $\varphi$ induces bijection $\varphi:VT_B \to VT_{B'}$ which projects to a map on the orbits $\varphi:VB \to VB'$.
				
				For each $v\in VB$, let $\Ab(v)$ be the subgraph of groups of $\Ab$ collapsed to $v$ in $\Bb$. Note that $G_v:= \pi_1(\Ab(v)) \cong \Bb_v$ and that $\Ab(v)$ has strictly fewer edges than $\Ab$. Let $\Hc_v$ be the set of edge stabilizers of $T_{A(v)}$. It is exactly the set of elements of $\Hc \backslash \Fc_{\min}$ contained in $G_v$. Since $\Ab$ is reduced, so is $\Ab(v)$. We apply induction to $(\Ab(v),G_v,\Hc_v)$ and conclude that $\varphi|_{G_v}$ maps $G_v$ to $G'_{\varphi(v)}$ and $\Hc_v$ to $\Hc'_{\varphi(v)}$.
				
				Finally we observe that $\Hc$ is the union of $\Fc_{\min}$ with the set of conjugates in $G$ of elements of $\cup_{v\in VB}\Hc_v$. The same statement holds for $\Hc'$. Since $\varphi$ maps $\Fc_{\min}$ to $\Fc'_{\min}$ and $\cup_{v\in VB}\Hc_v$ to $\cup_{v'\in VB'}\Hc_{v'}$, we conclude that $\varphi$ maps $\Hc$ to $\Hc'$.
			
				Since $\Ab$ is finite, there are finitely many conjugacy classes of edge stabilizers in $G$. The last assertion follows since $\Out_\Hc(G)$ is the subgroup of $\Out(G)$ which induces a trivial permutation of the conjugacy classes of elements of $\Hc$.
			\end{proof}
		
	\section{The structure of $\Out_\Hc(G)$}
	\label{sec:structure_out}
		
		Throughout this section, we fix a triple $(\Ab, G, \Hc)$ as follows:
		$\Ab$ is a finite reduced graph of groups with finite edge groups and finitely generated vertex groups with at most one end. $\Hc$ is a family of elliptic subgroups of $G=\pi_1(\Ab)$ containing each edge stabilizer. Suppose moreover that $\Ab$ does not consist of a single vertex (i.e.\ $G$ is not one-ended). The following two conditions will be used in the sequel:
		\renewcommand{\theenumi}{(C\arabic{enumi})}
		\renewcommand{\labelenumi}{\theenumi}
		\begin{enumerate}
			\item For any $v \in VT_A$ the set of $G_v$-conjugacy classes of elements in $\Hc_v$ is finite.\label{C1}
			\item The normalizer of any edge group in any vertex group $G_v$ containing it is finitely generated. \label{C2}
		\end{enumerate}
		\renewcommand{\theenumi}{\arabic{enumi}}
		\renewcommand{\labelenumi}{(\theenumi)}
		
		\subsection{The graph of groups $\Bb$}	
		\label{subsec:def_Bb}
		
		The set of edge stabilizers of $T_A$ is nonempty as $G$ is assumed to have more than one end. As edge stabilizers are finite, an edge stabilizer can never be conjugate to a proper subgroup of itself, so the set of edge stabilizers partially ordered by inclusion admits a minimal element. Let $G_0$ be such a minimal edge stabilizer, and let $\Ec$ be the set of edges of $\Ab$ whose edge stabilizer is conjugate to $G_0$. We set $\bar\Bb$ to be the graph of groups obtained by collapsing all edges of $\Ab$ not in $\Ec$.
		
		Let $V=V\bar B$. For each $v \in V$, let $\Ab(v)$ be the connected subgraph of groups of $\Ab$ collapsed to $v$ in $\Bb$. As not all edges of $\Ab$ are collapsed in $\bar \Bb$, each $\Ab(v)$ has fewer edges than $\Ab$. These graphs of groups will not be referred to in this section, but will be used in Sections \ref{sec:finite_pres} and \ref{sec:finite_out} to prove Theorems~\ref{theorem:out_fp} and \ref{theorem:finite_out} respectively.
		
		Observe that $\bar\Bb$ is a finite reduced graph of groups with fundamental group $G$ and Bass-Serre tree $T_{\bar B}$ such that:
		\begin{itemize}
			\item All edge stabilizers of $G$ are conjugate;
			\item No edge stabilizer is conjugate to a proper subgroup of itself.
		\end{itemize}
		
		\begin{remark}
			The remainder of Subsection \ref{subsec:def_Bb} can be applied to any graph of groups $\bar \Bb$ satisfying the two conditions just stated. In fact, all the statements of Section \ref{sec:structure_out} except Corollary~\ref{corollary:ker_tau_G0_finite} remain true if the edge groups of $\bar \Bb$ are assumed to satisfy the two above conditions and to have finite outer automorphism group.
		\end{remark}
		
		In order to fix notation we modify $\bar \Bb$ to a more symmetric graph of groups $\Bb$ without changing the set of elliptic subgroups nor the set of edge groups. Subdivide an edge of $\bar \Bb$ and call $b_0$ the new vertex. Now slide the beginning of each edge to $b_0$ whenever possible. This makes a graph of groups $\Bb$ having vertex set $\{b_0\} \cup V$ with the following properties: $b_0$ has a vertex group isomorphic to some edge group in $\bar\Bb$; for each vertex $v$ of $V$ the vertex group $\Bb_v$ properly contains each incoming edge group; all loop edges begin at $b_0$; all non-loop edges that begin at $b_0$ end in $V$ and vice versa; for any two edges $e \neq f$ having a common terminal vertex $v \in V$, the subgroups $\omega_e(\Bb_e)$ and $\omega_f(\Bb_f)$ are not conjugate in $\Bb_v$.
		
		Write the set of loop edges as $\{e_s \mid s\in S\epm\}$ in such a way that $e_s\inv = e_{s\inv}$. Write also the set of non-loop edges as $\{e_k \mid k \in K\epm\}$ in such a way that $e_k\inv = e_{k\inv}$ and that for each $k\in K$ the edge $e_k$ originates at $b_0$ and ends in $V$. We identify the vertex group of $b_0$ with $G_0$. We introduce the following notation until the end of the section: $\omega(e_k)$ is denoted by $v_k$ for each $k \in K$. For each $k \in K$, we write $G_k$ for the subgroup $[1,e_k,\Bb_{v_k},e_k\inv,1] \subset \pi_1(\Bb,b_0)$. Choosing a maximal tree in $B$, we identify $\Bb_v$ with the group $G_k$ for some $k$ such that $v_k=v$. The normalizer of $G_0$ in $G_k$ is written as $N_k$ and the normalizer of $G_0$ in $G$ is denoted by $N$. We further let $F = N/G_0$ and for $k \in K$ we let $F_k = N_k / G_0$. We do not distinguish between $s \in S$ and the corresponding element $[1,e_s,1] \in \pi_1(\Bb,b_0)$.
		
		Observe that the normalizer of $G_0$ in $G$ has a very simple decomposition $\bar \Nb$ as an amalgam of the groups $N_k$ for $k\in K$ and some mapping tori along $G_0$. Similarly, the group $F=N/G_0$ inherits a decomposition $\bar \Fb$ obtained from $\bar \Nb$ by taking the quotient of every edge and vertex group by $G_0$. Thus edge groups of $\bar \Fb$ are trivial and $\bar \Fb$ is decomposition of $F$ as a free product $F = (\ast_{k \in K} (F_k)) \ast F(S)$ (where $F(S)$ denotes the free group on the set $S$). The decompositions $\Bb$ and $\bar \Nb$ are described in Figure \ref{fig:rose}.
		
		It could happen that $\bar \Nb$ is not minimal, as there can be some $k$ such that $N_k = G_0$. This is represented by a dashed edge in Figure \ref{fig:rose}. We define the graph of groups $\Nb$ and $\Fb$ by removing these vertices and edges from the decomposition $\bar \Nb$ as well as the corresponding ones from $\bar \Fb$. Let $I$ be the set of $i\in K$ such that $N_i = G_0$, and let $J = K \backslash I$ be the other indices. 
		
		\begin{figure}[htb]
			\begin{center}
				\input{Rose.pstex_t}
				\hspace{20mm}
				\input{Normalizer.pstex_t}
			\end{center}
			\caption{The decomposition $\Bb$ of $G$ and $\bar \Nb$ of $N = \Norm_G(G_0)$}
			\label{fig:rose}
		\end{figure}
		
		\begin{lemma} \label{lemma:vertex_stable}
			The family of vertex stabilizers of $T_B$ is preserved by $\Out_\Hc(G)$.
		\end{lemma}
		\begin{proof}
			By construction of $\Bb$, all edge stabilizers of $G$ acting on $T_B$ are conjugate to $G_0$. Let $\hat \varphi \in \Out_\Hc(G)$ be a relative automorphism and $\varphi$ be a representative of $\hat \varphi$. Since $G_0$ lies in $\Hc$ the automorphism $\varphi$ preserves the conjugacy class of $G_0$, and also preserves the family of subgroups of conjugates of $G_0$.
			
			Recall that $G_0$ is a minimal element among the family of edge stabilizers of $G$ acting on $T_A$ and that vertex groups of $\Ab$ have at most one end. Therefore if a subgroup of $G$ does not split over a subgroup of a conjugate of $G_0$ then it is contained in a vertex stabilizer of $T_B$. So vertex stabilizers of $T_B$ are either conjugates of $G_0$ or maximal subgroups of $G$ which do not split over a subgroup of a conjugate of $G_0$. Since $\varphi$ preserves the conjugacy class of $G_0$ it must preserve the family of vertex stabilizers.
		\end{proof}
		
		\subsection{The maps $\rho$} \label{subsec:rho}
		For every subgroup $H$ of a group $\Gamma$, let $\Out_H(\Gamma)$ be the subgroup of the outer automorphism group preserving the conjugacy class of $H$. Then there is a map $\rho_H : \Out_H(\Gamma) \to \Out(\Norm_\Gamma(H))$ defined as follows. Fix an element $\hat\alpha$ in  $\Out_H(\Gamma)$. Choose a representative $\alpha$ of $\hat\alpha$ in $\Aut(\Gamma)$ such that $\alpha(H)=H$. This representative is unique up to right multiplication by an inner automorphism normalizing $H$. Moreover, $\alpha(\Norm_\Gamma H) = \Norm_\Gamma H$. Define $\rho_H(\hat\alpha)$ to be the element of $\Out(\Norm_\Gamma(H))$ represented by the restriction of $\alpha$ to $\Norm_\Gamma H$. By the discussion above, this map is well defined. $\rho_H$ is easily checked to be a homomorphism.
		
		Having made this observation, we can now define $\Out^0(G)$. By Lemma~\ref{lemma:vertex_stable} the group $\Out_\Hc(G)$ permutes the (finite) set of conjugacy classes of vertex stabilizers of $\Bb$. Let $\Out'_\Hc(G)$ be the subgroup of $\Out_\Hc(G)$ which induce the trivial permutation of this set. Moreover Lemma~\ref{lemma:normvertexgp} implies that the vertex stabilizer $G_v$ is self-normalized in $G$ for each $v \in V$, so there are maps $\rho_v : \Out'_\Hc(G) \to \Out(G_v)$. The image of $\rho_v$ preserves the family $\Hc_v$, but does not necessarily preserve the $G_v$-conjugacy class of each element. Let $\Out^0(G)$ be the largest subgroup of $\Out_\Hc(G)$ which maps to $\Out_{\Hc_v}(G_v)$ for each $v\in V$, i.e.\ the subgroup of $\Out_\Hc(G)$ which preserves the $G_v$-conjugacy class of each element of $\Hc_v$. If condition \ref{C1} is fulfilled, i.e.\ if there are finitely many such conjugacy classes for each $v \in VA$, $\Out^0(G)$ has finite index in $\Out_\Hc(G)$. 
		
		Putting all the maps $\rho_v$ together, we define the map 
		\[\rho_\Bb : \Out^0(G) \stackrel{\prod \rho_v}{\longrightarrow} \prod_{v \in V} \Out_{\Hc_v}(G_v)\]
		
		Let $\Nc$ be the set of vertex stabilizers of $N$ acting on $T_N$, i.e.\ the family of subgroups conjugate to either $G_0$ or to one of the $N_j$. Let $\Fc$ be set of vertex stabilizers of $F$ acting on $T_F$. Remark that $\Fc$ is precisely the image of $\Nc$ in $F$ under the projection $\pi:N \to F$. Similarly to above, we have maps $\rho'_j:\Out_\Nc(N) \to \Out_{G_0}(N_j)$ and $\rho''_j: \Out_\Fc(F) \to \Out(F_j)$ and we again combine these maps to form
		\[ \rho_\Nb : \Out_\Nc(N) \stackrel{\prod \rho'_j}{\longrightarrow} \prod_{j \in J} \Out_{G_0}(N_j) \quad \quad ; \quad \quad \rho_\Fb : \Out_\Fc(F) \stackrel{\prod \rho''_j}{\longrightarrow} \prod_{j \in J} \Out(F_j)\]
		Since $\Out^0(G)$ preserves in the conjugacy class of $G_0$ and the conjugacy class of each $G_j$, we define the map $\sigma : \Out^0(G) \to \Out_\Nc(N)$ as above. For each $k \in I \cup J$ there is a map $\sigma_k : \Out_{\Hc_{v_k}}(G_{v_k}) \to \Out_{G_0}(N_k)$. Define $\sigma_I = \prod_{i \in I} \sigma_i$ and $\sigma_J = \prod_{j \in J} \sigma_j$.
		
		Since any automorphism representing an element of $\Out_\Nc(N)$ must fix $G_0$, the projection $\pi:N \to N/G_0$ induces a map $\tau : \Out_\Nc(N) \to \Out_\Fc(F)$, and maps $\tau_j : \Out_{G_0}(N_j) \to \Out(F_j)$. Define $\tau_J = \prod_{j \in J} \tau_j$.
		
		Let $\Wc_\Bb$, $\Wc_\Nb$ and $\Wc_\Fb$ denote the kernels of the maps $\rho_\Bb$, $\rho_\Nb$ and $\rho_\Fb$ respectively. These groups are exactly those outer automorphisms $\hat \varphi$ such that any representative $\varphi$ of $\hat \varphi$ acts on each vertex stabilizer as a conjugation by some element in the corresponding group. Since the squares at the right of the diagram in Figure \ref{fig:CVS_groups} are commutative, the maps $\sigma$ and $\tau$ induce maps $\sigma^* : \Wc_\Bb \to \Wc_\Nb$ and $\tau^* : \Wc_\Nb \to \Wc_\Fb$ on the kernels.
		
		\begin{figure}[htb]
			\[ \begin{xy}
		  \xymatrix{
		     \Wc_\Bb \ar@{^{(}->}[r] \ar[d]_{\sigma^*} & \Out^0(G) \ar[r]^{\rho_\Bb} \ar[d]_{\sigma}    &  \prod_{v \in V} \Out_{\Hc_v}(G_v)  \ar[d]^{\sigma_J} \ar[rd]^{\sigma_I}  \\
		     \Wc_\Nb \ar@{^{(}->}[r] \ar[d]_{\tau^*} & \Out_\Nc(N) \ar[r]^{\rho_\Nb} \ar[d]_{\tau} &   \prod_{j \in J} \Out_{G_0}(N_j) \ar[d]^{\tau_J}& \prod_{i \in I} \Out(G_0)\\
		     \Wc_\Fb \ar@{^{(}->}[r] & \Out_\Fc(F)  \ar[r]^{\rho_\Fb} &   \prod_{j \in J} \Out(F_j)
		  	}
			\end{xy}\]
			\caption{The diagram defining the maps $\sigma^*$ and $\tau^*$.}
			\label{fig:CVS_groups}
		\end{figure}
		
		The following lemma extends, under certain conditions, an automorphism of a vertex stabilizer to a group acting on a tree, and will be our main tool for constructing automorphisms.
		\begin{lemma}[{\cite[Proposition 2.1]{Levitt}}] \label{lemma:extension}
			Let $\Gamma$ be a group acting on a tree $T$ and let $v$ be a vertex of $T$. Let $\varphi$ be an automorphism of $\Gamma_v$ that acts by conjugation (in $\Gamma_v$) on each edge stabilizer contained in $\Gamma_v$. Then $\varphi$ extends to an automorphism $\widetilde \varphi$ of $\Gamma$. Moreover, one can choose $\widetilde \varphi$ so that it acts by conjugation on each vertex stabilizer $\Gamma_w$ for any vertex $w$ not in the $\Gamma$-orbit of $v$.
		\end{lemma}
		\begin{proof}
			Let $\Cb$ be the quotient graph of groups and $u$ be the projection of $v$ in $\Cb$. For each edge $e \in EC$ such that $\alpha(e)=u$, choose $\gamma_e \in \Cb_u$ such that $\varphi|_{\Cb_e} = i_{\gamma_e}|_{\Cb_e}$. For all other edges $f\in EC$, let $\gamma_f=1$. Letting $X$ be the set of $\Cb$-paths, we define
			\[ \widetilde \varphi: X \to X : a_0,e_1,a_1,\cdots,e_n,a_n \mapsto \bar a_0,e_1,\bar a_1,\cdots,e_n,\bar a_n \]
			where
			\[ \bar a_i=\begin{cases} a_i & \text{if } a_i \notin \Cb_{u} \\ \gamma_{e_i}\inv \varphi(a_i) \gamma_{e_{i+1}} & \text{if } a_i \in \Cb_{u} \end{cases} \ \ \ \ (\gamma_{e_0} = \gamma_{e_{n+1}} := 1)\]
			Routine computations show that $\widetilde \varphi$ induces a well defined endomorphism of $\Gamma=\pi_1(\Cb,u)$. Moreover, if we define $\widetilde \varphi'$ similarly by extending $\varphi'=\varphi\inv$ and $\gamma_e' := \sigma\inv(\gamma_e\inv)$ it can be checked that $\widetilde \varphi \widetilde \varphi' = \widetilde \varphi' \widetilde \varphi = \id_\Gamma$ so that $\widetilde \varphi$ is an automorphism of $\Gamma$. 
			
			If $w$ is a vertex of $T$ with projection $x \neq u$ in $\Cb$, the vertex stabilizer $\Gamma_w$ can be written as $[p\Cb_{x}p\inv]$ for some $\Cb$-path $p$ starting at $u$ and ending at $x$. It is clear that the restriction of $\widetilde \varphi$ to $\Gamma_w$ is conjugation by $[\widetilde \varphi(p) p\inv]$.
		\end{proof}
		\begin{remark}
			If we are given the trivial automorphism of $\Gamma_v$ and choose $\gamma_e=1$ for all edges $e \in EC$ except for one edge $f$ starting at $u$ for which we take $\gamma_f \in \Centr_{\Cb_u}\alpha_f(\Cb_f)$, then the automorphism constructed in the last lemma is a Dehn twist along the edge $f$.
		\end{remark}
		\begin{corollary} \label{corollary:im_sigma_im_rho}
			The map $\sigma^*$ is surjective. If $\Out(G_0)$ is finite, the image of $\rho_\Bb$ has finite index in $\prod_{v \in V} \Out_{\Hc_v}(G_v)$.
		\end{corollary}
		\begin{proof}
			We introduce an auxiliary graph of groups $\Bb'$ obtained from $\Bb$ as follows (see Figure \ref{fig:bprime}): collapse all loop edges of $\Bb$ and slide $N_k$ along each edge $e_k$, thus setting $N$ to be the vertex group of $b_0$. 
			
			Let us show that $\sigma^*$ is surjective. If $\hat \varphi$ is an outer automorphism in $\Kc_\Nb$, any representative $\varphi$ of $\hat \varphi$ acts as a conjugation on each $N_k$ so one can apply Lemma~\ref{lemma:extension} and extend $\varphi$ to an automorphism $\psi$ of $G$ such that $\psi$ acts as a conjugation on each $G_v$, so that $\psi$ is in $\Kc_\Bb$ and $\sigma^*(\hat \psi) = \hat \varphi$.
			
			If $\Out(G_0)$ is finite the subgroup $\Out^*(G_v)$ of $\Out_{\Hc_v}(G_v)$ which acts as a conjugation on each edge stabilizer contained in $G_v$ is of finite index in $\Out_{\Hc_v}(G_v)$. Applying Lemma~\ref{lemma:extension} for each $v\in V$ yields a lift of any element in $\prod_{v\in V} \Out^*(G_v)$, which finishes the proof.
		\end{proof}
		\begin{figure}[hbt]
				\begin{center}
					\input{bprime.pstex_t}
				\end{center}
				\caption{The decomposition $\Bb'$ of $G$}
				\label{fig:bprime}
			\end{figure}

		\subsection{The kernels of $\sigma^*$ and $\tau^*$}
		
		Let $\Bb'$ be the graph of groups defined in the proof of Corollary~\ref{corollary:im_sigma_im_rho}. We describe the kernel of $\sigma^*$ in terms of Dehn twists in that splitting of $G$. For $k\in K$ and $g_k \in \Centr_G(N_k)$, define the automorphism $D^k_{g_k}$ of $G$ as the Dehn twist of $\Bb'$ along the edge $e_k$ and with element $g_k$. Clearly, $\widehat{D^k_{g_k}}$ lies in the kernel of $\sigma^*$, and $D^k_{g_k}$ commutes with $D^{k'}_{g_{k'}}$ if $k$ is different from $k'$. This defines a homomorphism 
		\[ q : \prod_{k\in K} \Centr_G(N_k) {\to} \ker \sigma^* : (g_k)_{k \in K} \mapsto \prod_{k \in K} \widehat{D^k_{g_k}} \]
		\begin{lemma}\label{lemma:ker_sigma}
			The kernel of $\sigma^*$ fits into the exact sequence
			\[1 \to \Centr(G) \to \Centr(N) \times \prod_{v\in V} \Centr(G_v) \stackrel{p}{\to} \prod_{k\in K} \Centr_G(N_k) \stackrel{q}{\to} \ker \sigma^* \to 1\]
			where $p$ embeds $\Centr(N)$ diagonally into $\prod_{k\in K} \Centr_G(N_k)$ and embeds $\Centr(G_v)$ diagonally into $\prod_{v_k=v} \Centr_G(N_k)$.
		\end{lemma}
		\begin{proof}
			Let $\hat\varphi$ be an outer automorphism in the kernel of $\sigma^*$. By definition of $\sigma^*$ we can choose a representative $\varphi$ of $\hat \varphi$ that restricts to the identity on $N$. For each $k,k'$ such that $v_k=v_{k'}$ we let $t_{kk'} = [1,e_k,1,e_{k'}\inv,1] \in G= \pi_1(\Bb',b_0)$. Note that if $k, k'\in K$ are such that $v_k=v_{k'}$, we have $G_{k'} = t_{k'k} G_k t_{k'k}\inv$. Therefore, if $g_k\in G$ is such that $\varphi|_{G_k} = i_{g_k}|_{G_k}$ then $\varphi|_{G_{k'}} = i_{\varphi(t_{k'k})g_k}|_{G_{k'}}$. Moreover, since $\varphi$ restricts to the identity on $N$ the element $g_k$ must belong to $\Centr_G(N_k)$. Fix a tuple $(g_k)_{k\in K}$ such that 
			\begin{align}
				\varphi|_{G_k} &= i_{g_k}|_{G_k} \label{varphi1} \\
				g_{k'}&=\varphi(t_{k'k})g_k \text{ for each } k,k' \text{ with } v_k=v_{k'} \label{varphi2}
			\end{align} 
			Notice that the set \[\bigcup_{k \in K} G_k \bigcup N \bigcup \{t_{kk'} | v_k=v_{k'}\}\] generates $G= \pi_1(\Bb',b_0)$. Therefore the tuple $(g_k)_{k\in K}$ defines $\varphi$ from the fact that $\varphi$ restricts to the identity on $N$ and from properties \eqref{varphi1} and \eqref{varphi2}. In conclusion, $\varphi = \prod_{k\in K} D^k_{g_k}$, and the kernel of $\sigma^*$ can be viewed as a quotient of $\prod_{k\in K} \Centr_G(N_k)$.
			
			Let now $(g_k) \in \prod_{k\in K} \Centr_G(N_k)$ be a tuple in the kernel of $q$. In other words $\varphi = \prod_{k\in K} D^k_{g_k}$ is an inner automorphism $i_n \in \Inn(G)$. Since $\varphi$ acts as the identity on $N$, the element $n$ must lie in the centralizer of $N$, which equals the center of $N$ as $G_0 \subset N$. Hence $\tilde\varphi:=\prod_{k \in K} D^k_{n\inv} \varphi$ is the identity on $G$. Therefore $\tilde{g_k}:=n\inv g_k$ lies in the center of $G_k$ for each $k$. It remains to show that if $v_k = v_{k'}$ then $\tilde{g_k} = \tilde{g_{k'}}$. Since $\varphi$ is the identity on $G$, we have
			$\varphi(t_{kk'}) = [1,e_k,\tilde{g_k} \tilde g_{k'}\inv,e_{k'}\inv,1] = [1,e_k,1,e_{k'}\inv,1] = t_{kk'}$
			which implies that $\tilde{g_k}\tilde g_{k'}\inv = 1$.
			
			The kernel of $p$ is exactly those elements of $\Centr(N)$ which commute with all elements of the generating set of $G$ described above. Therefore, the kernel of $p$ is isomorphic to the center of $G$.
		\end{proof}
		\begin{remark}
			Note that for each $k \in K$, either $N_k=G_0$ in which case $\Centr_G(N_k) = \Centr_G(G_0)$, or $N_k \neq G_0$ and $\Centr_G(N_k) = \Centr(N_k)$ by Lemma~\ref{lemma:normvertexgp}. Hence the exact sequence in Lemma~\ref{lemma:ker_sigma} can be rewritten as
			\[1 \to \Centr(G) \to \Centr(N) \times \prod_{v\in V} \Centr(G_v) \stackrel{p}{\to} \prod_{i\in I} \Centr_G(G_0) \times \prod_{j \in J} \Centr(N_j) \stackrel{q}{\to} \ker \sigma^* \to 1\]
		\end{remark}
		
		Let $\Tc$ be the subgroup of the kernel of $\tau^*$ consisting of those outer automorphisms $\hat \psi$ that have a representative $\psi$ that restricts to the identity on $G_0$ and that induce the identity on the quotient $N/G_0$. It is clear that if $\Out(G_0)$ is finite, $\Tc$ has finite index in the kernel of $\tau^*$. Let $K_j$ denote the subgroup of $N_j$ of those elements that act trivially by conjugation on $N_j/G_0$.
		
		If $n_j$ is an element of $\Centr_{K_j}(G_0)$, let $D^j_{n_j}$ be the automorphism acting on $N_j$ as conjugation by $n_j$ and as the identity on $N_j'$ for $j'$ different from $j$ and on $s \in S$. If $n_s$ is an element of the center of $G_0$ and $s \in S$, let $D^s_{n_s}$ be the automorphism of $N$ acting as the identity on $N_j$ for each $j \in J$, fixing $s'$ for each $s'$ different from $s$ and mapping $s$ to $n_ss$. Clearly, $\widehat{D^l_{n_l}}$ lies in $\Tc$ for each $l \in J \cup S$, and $D^l_{n_l}$ commutes with $D^{l'}_{n_{l'}}$ if $l$ is different from $l'$. Therefore there is a map 
		\[g:\prod_{s \in S} \Centr(G_0) \times \prod_{j\in J} \Centr_{K_j}(G_0) \to \Tc : (n_l)_{l \in S \cup J} \mapsto \prod_{l \in S\cup J} \widehat{D^l_{n_l}} \]
		
		\begin{lemma}\label{lemma:ker_tau}
			$g$ is surjective and $(1)_{s \in S} \times \prod_{j \in J} \Centr(N_j)$ lies in the kernel of $g$.
		\end{lemma}
		\begin{proof}
			Let $\hat \psi$ be an outer automorphism of $N$ in $\Tc$, and choose a representative $\psi$ of $\hat \psi$ that restricts to the identity on $G_0$ and such that $\tau(\psi)$ that acts as the identity on $F$. Choose elements $n_j$ such that $\psi|_{N_j} = i_{n_j}|_{N_j}$. Since $\tau^*(\psi)$ acts as the identity on $F$, this means that $n_j \in \Centr_{N_j}(G_0) \cap K_j = \Centr_{K_j}(G_0)$. It also implies that there exist elements $n_s$ in the center of $G_0$ such that $\psi(s) = n_ss$. Hence $\psi$ is the product of the automorphisms $D^l_{n_l}$ where $l$ ranges through $J \cup S$, and $g$ is surjective.
			
			If $n_j \in \Centr(N_j)$ then it is clear that that $D^j_{n_j} = \id_N$. Hence $(1)_{s \in S} \times \prod_{j \in J} \Centr(N_j)$ lies in the kernel of $g$.
		\end{proof}
		\begin{corollary}\label{corollary:ker_tau_G0_finite}
			If $G_0$ is finite and $N_j$ is finitely generated for each $j \in J$, then the kernel of $\tau^*$ is finite.
		\end{corollary}
		\begin{proof}
			In light of Lemma~\ref{lemma:ker_tau}, it suffices to show that if $G_0$ is finite and $N_j$ is finitely generated, then $\Centr(N_j)$ has finite index in $\Centr_{K_j}(G_0)$. Let $n_1,...,n_p$ be a finite generating set for $N_j$. Observe that $\Centr_{K_j}(G_0) = \Centr_{N_j}(G_0) \cap K_j$. Moreover, if $n \in N_j$ and $k \in K_j$, then the commutator $[n,k]:=nkn\inv k\inv$ lies in $G_0$ by definition of $K_j$. For each $1\leq i\leq p$, define a map $\beta_i: \Centr_{K_j}(G_0) \to G_0$ sending $k$ to $[n_i,k]$. This is a homomorphism since 
			\[\beta_i(k)\beta_i(k') = n_i k n_i\inv k\inv (n_i k' n_i\inv {k'}\inv) = n_i k n_i \inv (n_i k' n_i \inv {k'}\inv) k\inv = \beta_i(kk')\]
			The center of $N_j$ is a subgroup of $\Centr_{K_j}(G_0)$ and is exactly the intersection of the kernels of all $\beta_i$, which are all finite index subgroups of $\Centr_{K_j}(G_0)$ since $G_0$ is finite. Hence $\Centr(N_j)$ has finite index in $\Centr_{K_j}(G_0)$ for each $j \in J$.
		\end{proof}
		
		If $G_0$ is not finite or there is an $N_j$ which is not finitely generated, there is still some control over $\Tc$, hence over $\ker \tau^*$ if $\Out(G_0)$ is finite. Indeed, the following proposition holds:
		\begin{proposition}[{\cite[Proposition 3.1]{Levitt}}]\label{proposition:seq_ker_tau} If $\Nb$ is minimal and not a mapping torus (i.e.\ $|J|+|S|\geq 2$) there is an exact sequence
			\[1 \to \Centr(N) \to \Centr(G_0) \times \prod_{s \in S} \Centr(G_0) \times \prod_{j \in J} \Centr(N_j) \stackrel{f}{\to} \prod_{s \in S\epm} \Centr(G_0) \times \prod_{j\in J} \Centr_{K_j}(G_0) \stackrel{g}{\to} \Tc \to 1.\]
			\qedhere
		\end{proposition}
		
		\subsection{Labellings and the image of $\tau^*$}
		
		Let $\chi:F \to \Out(G_0)$ be the homomorphism defined by $f \mapsto \widehat{i_{\tilde f}|_{G_0}}$ for some lift $\tilde f \in N$ of $f$. In order to describe the image of $\tau^*$ we introduce the labelling of an outer automorphism $\hat \psi \in \Wc_\Fb$ as follows. First, choose a representative $\psi \in \hat \psi$. For each $j\in J$, choose an element $f^\psi_j$ of $F$ such that $\psi|_{F_j}=i_{f^\psi_j}|_{F_j}$ and let $f^\psi_s = \psi(s)$ for each $s \in S$. Such a choice of $f^\psi_l$ for $l \in J \cup S$ is called a \emph{defining tuple} for $\psi$. The \emph{labelling $\lab(\hat \psi)$ of $\hat \psi$} is the subset of $\prod_{l \in J \cup S} \Out(G_0)$ consisting of all tuples $(\chi(f^\psi_l))_{l \in J \cup S}$ obtained from all defining tuples $f^\psi_l$ for all representatives $\psi$ of $\hat\psi$.
		
		\begin{lemma} \label{lemma:lab_intersection}
			For any outer automorphisms $\hat \psi, \hat \psi'$ in $\Wc_\Fb$, either the labellings $\lab(\hat \psi)$ and $\lab(\hat \psi')$ do not intersect or they are equal.
		\end{lemma}
		\begin{proof}
			Suppose that $\psi$ is a representative of $\hat\psi$, and let $(f_l)_{l \in J \cup S}$ be a defining tuple for $\psi$. Then $\tilde f_l$ is another defining tuple for $\psi$ if and only if $f_j \inv \tilde f_j$ lies in the center of $F_j$ for each $j\in J$ and $f_s = \tilde f_s$.
			Moreover, all other representatives of $\hat\psi$ are of the form $\bar\psi = i_g \psi$ for some $g\in F$, and in that case $\bar f_j := g f_j$ and $\bar f_s := g f_s g\inv$ is a defining tuple for $\bar \psi$.
			
			Therefore, if $(\chi_l)_{l\in J \cup S}$ is an tuple in a labelling $\lab(\hat \psi)$ then all other tuples in $\lab(\hat \psi)$ are obtained by performing all finite sequences of the following operations:
			\begin{itemize}
				\item Replace $(\chi_l)_{l\in J \cup S}$ by $((\chi(g)\chi_j)_{j \in J},(\chi(g)\chi_s\chi(g\inv))_{s \in S})$ For some $g\in F$.
				\item Replace the $j\text{th}$ entry of $(\chi_l)_{l\in J \cup S}$ by $\chi_j\chi(g_j)$ for some element $g_j$ in the center of $F_j$.
			\end{itemize}
			This finishes the proof as this way to obtain all tuples in a labelling from one given tuple does not depend on $\hat \psi$.
		\end{proof}
		
		\begin{lemma} \label{lemma:stab_lab_id}
			Any element of $\Wc_\Fb$ having the same labelling as the identity lies in the image of $\tau^*$.  
		\end{lemma}
		\begin{proof}
			Suppose that $\hat\psi$ has the same labelling as the identity. Thus there is a representative $\psi$ of $\hat \psi$ and elements $f_l \in F$ such that $\chi(f_j) = \widehat{\id|_{G_0}}$ and such that $\chi(f_s) = \widehat{i_s|_{G_0}}$. Thus if $n_l$ are lifts in $N$ of $f_l$, then there are elements $g_l$ in $G_0$ such that $i_{n_jg_j\inv}|_{G_0} = \id|_{G_0}$ and such that $i_{n_sg_s\inv}|_{G_0} = i_s|_{G_0}$. Define a homomorphism $\varphi:N\to N$ on a generating set of $N$ as follows:
			\[ \varphi|_{N_j} = i_{n_jg_j\inv}|_{N_j} \text{ and } \varphi(s)=n_sg_s\inv\]
			This choice ensures that the image of the generating set satisfy the relations presenting $N$ so that this map extends to an endomorphism of $N$. Denoting by $\pi$ the projection of $N$ on $N/G_0 = F$, it follows from the construction of $\phi$ that $\pi \circ \varphi = \psi \circ \pi$. Therefore, $\pi(\ker \varphi) \subset \ker \psi = \{\id\}$ and $\pi(\image \varphi) \supset \image \psi = F$. Since $\varphi$ is injective on $G_0$, it follows that $\varphi$ is an automorphism of $N$, and $\tau^*(\hat \varphi) = \hat\psi$ by construction.
		\end{proof}
		\begin{proposition} \label{proposition:action_well_def}
			$\Wc_\Fb$ acts from the right on the set of labellings by the formula \[\lab(\hat \psi_1) \cdot \hat \psi_2 := \lab(\hat \psi_1 \hat \psi_2)\]
		\end{proposition}
		\begin{proof}
			Assume for a moment that this is well-defined. Then the formula indeed defines an action since
			\[(\lab(\hat \psi_1) \cdot \hat\psi_2) \cdot \hat\psi_3 = \lab(\hat \psi_1 \hat\psi_2 \hat\psi_3) = \lab(\hat \psi_1) \cdot (\hat\psi_2 \hat \psi_3)\]
			
			Let us now show that the formula is well defined. In other words, we have to show that if $\hat \psi, \hat \psi'$ and $\hat \psi^*$ are outer automorphisms in $\Wc_\Fb$ such that $\hat \psi$ and $\hat \psi'$ have the same labelling, then $\hat \psi \hat \psi^*$ and $\hat \psi' \hat \psi^*$ have the same labelling.
			Let $\psi$, $\psi'$, $\psi^*$ and $f_l$,$f_l'$,$f_l^*$ be defining tuples for $\hat \psi$, $\hat \psi'$, $\hat \psi^*$ respectively, and suppose that $\chi(f_l) = \chi(f_l')$ for each $l$. The automorphisms $\bar \psi := \psi \psi^*$ and $\bar \psi' := \psi' \psi^*$ are representatives of $\hat \psi \hat \psi^*$ and $\hat \psi' \hat \psi^*$ respectively. Computation shows that the elements $\bar f_j := \psi(f_j^*) f_j$ and $\bar f_s := \psi(f_s^*)$ are such that $\bar \psi|_{F_j} = i_{\bar f_j}|_{F_j}$ and such that $\bar \psi(s) = \bar f_s$. Similarly the elements $\bar f_j' := \psi'(f_j^*) f_j'$ and $\bar f_s' := \psi'(f_s^*)$ are a defining tuple for $\bar \psi'$.
			
			Let us show that $\chi(\psi(g)) = \chi(\psi'(g))$ for any $g\in F$. It suffices to show it for a generating set of $F$. If $g$ lies in $F_j$ for some $j \in J$, then 
			\[\chi(\psi(g)) = \chi(f_j)\chi(g)\chi(f_j\inv) = \chi(f_j')\chi(g)\chi({f'}_j\inv) = \chi(\psi'(g)).\]
			If $g = s$ for some $s \in S$, then
			\[\chi(\psi(g)) = \chi(f_s) = \chi(f_s') = \chi(\psi'(g)).\]
			It follows that $\chi(\bar f_l) = \chi (\bar f_l')$. Thus the labellings of $\hat \psi \hat \psi^*$ and $\hat \psi' \hat \psi^*$ intersect, hence they must coincide by Lemma~\ref{lemma:lab_intersection}.
		\end{proof}
		\begin{corollary} \label{corollary:tau_virt_epi}
			The image of $\tau^*$ has finite index in $\Wc_\Fb$.
		\end{corollary}
		\begin{proof}
			Since $\Out(G_0)$ is finite, then so is the set of labellings. An automorphism $\hat \varphi$ has the same labelling as the identity if and only if it is in the stabilizer of $\lab(\id)$ for the action defined in proposition \ref{proposition:action_well_def}. Hence, it follows from Lemma~\ref{lemma:stab_lab_id} that the image of $\tau^*$ has finite index in $\Wc_\Fb$.
		\end{proof}
		
		\subsection{Conclusion} A short sequence $A \stackrel{\lambda}{\hookrightarrow} C \stackrel{\mu}{\rightarrow} B$ is \emph{virtually exact} if $\lambda(A)$ is a finite index subgroup of $\ker \mu$ and if $\mu(C)$ has finite index in $B$. The results of this section can be summarized by the following
			\begin{proposition} \label{proposition:main_tech}
				If we let $D=\ker \sigma^*$ and suppose condition \ref{C2} holds, then the following short sequences
				\[\Wc_\Bb \hookrightarrow \Out^0(G) \stackrel{\rho_\Bb}{\longrightarrow}  \prod_{v \in VB} \Out_{\Hc_v}(G_v) \text{ and } D \hookrightarrow \Kc_\Bb \stackrel{\tau^*\circ \sigma^*}{\longrightarrow} \Kc_\Fb\]
				 are virtually exact. If condition \ref{C1} holds, $\Out^0(G)$ has finite index in $\Out_\Hc(G)$.
			\end{proposition}
			\begin{proof}
				$\Wc_\Bb = \ker \rho_\Bb$ by definition. $\rho_\Bb$ is virtually surjective by the second part of Corollary~\ref{corollary:im_sigma_im_rho}. Condition \ref{C2} ensures that the normalizer of any edge group in a vertex group is finitely generated. Thus Corollary~\ref{corollary:ker_tau_G0_finite} applies and $\ker\tau^*$ is finite. Therefore $D = \ker \sigma^*$ has finite index in $\ker (\tau^* \circ \sigma^*)$. Finally $\tau^* \circ \sigma^*$ is virtually surjective by combining the first part of Corollary~\ref{corollary:im_sigma_im_rho} with Corollary~\ref{corollary:tau_virt_epi}. 
				
				The last assertion was mentioned when defining $\Out^0(G)$.
			\end{proof}
			Remark that Lemma~\ref{lemma:ker_sigma} describes the structure of the group of Dehn twists $D$. A presentation of $\Wc_\Fb$ is discussed in Theorem~\ref{theorem:F-R}.
		
	\section{Finite presentation for $\Out(G)$} \label{sec:finite_pres}
		
		Throughout this section, we use without mentioning it the following basic facts about finite presentations. If $H$ is a finite index subgroup of $G$, then $H$ is finitely presented if and only if $G$ is. If $G$ is a finitely presented group and $N$ is a finitely generated normal subgroup, then $G/N$ is finitely presented. If $Q=G/N$ and both $N$ and $Q$ are finitely presented, then so is $G$.
		
		In view of Proposition \ref{proposition:main_tech}, a key step in understanding presentations of automorphism groups of accessible groups is to understand presentations of the automorphism group of a free product in the guise of the group $\Wc_\Fb$. We forget a moment about Section~\ref{sec:structure_out} and suppose we are given a free product $F = (\ast_{j \in J} F_j) \ast F(S)$ where $S$ and $J$ are finite, and $F(S)$ is the free group on the set $S$. Consider the following subgroups of $\Aut(F)$:
		\begin{itemize}
			\item $\Aut^0(F)$ is the group of automorphisms which maps each $F_j$ to a conjugate of itself (setwise). 
			\item The group $\Aut_f(F)$ of \emph{factor automorphisms} is the subgroup of $\Aut^0(F)$ which maps each $F_j$ to itself, and maps each $s \in S$ to either $s$ or $s\inv$.
			\item $\Aut_i(F)$ is the subgroup of $\Aut^0(F)$ which acts as a conjugation on each $F_j$.
		\end{itemize}
		We briefly comment on these groups. In the special case where the factors $F_j$ are freely indecomposable, not infinite cyclic and pairwise non-isomorphic, $\Aut^0(F)$ is the full automorphism group $\Aut(F)$. In general, $\Aut^0(F) \subset \Aut(F)$ is the preimage of $\Out_\Fc(F) \subset \Out(F)$ where $\Fc$ is the family of conjugates of the factors $F_j$. The structure of $\Aut_f(F)$ is clear. Indeed 
		\[\Aut_f(F) \cong \prod_{j \in J} \Aut(F_j) \times \prod_{s \in S} \Aut(\Zb).\]
		Finally, if $F$, $F_j$ and $S$ are as in Section~\ref{sec:structure_out} the group $\Aut_i(F) \subset \Aut(F)$ is the preimage of $\Wc_\Fb \subset \Out(F)$ so that $\Wc_\Fb \cong \Aut_i(F) / \Inn(F)$.
		
		We sketch the peak reduction method as applied by Gilbert \cite{Gilb} to find a finite presentation of $\Aut^0(F)$.
		
		Gilbert defines a set $\Wh \subset \Aut_i(F)$ of \emph{Whitehead automorphisms} of $F$ which generalize Whitehead's generating set of the automorphism group of a free group. By the work of Fouxe-Rabinovitch \cite{FR1} the set $\Omega = \Aut_f(F) \cup \Wh$ generates $\Aut^0(F)$. A presentation of $\Aut^0(F)$ on the set of generators $\Omega$ is found as follows. One defines a complexity $\left| \  \right| : \Aut^0(F) \to \Nb$ such that $| z |$ is minimal if and only if $z \in \Aut_f(F)$. The crucial peak reduction lemma asserts that if $z \in \Aut^0(F)$ and $w_1, w_2 \in \Omega$ such that $|zw_1| \leq |z| \geq |zw_2|$ with at least one inequality being strict, then there is a set of generators $w'_1, w'_2, \ldots, w'_n$ such that $|zw_1w'_1|, |zw_1w'_1w'_2|, \ldots ,|zw_1w'_1\ldots w'_{n-1}|$ are all strictly smaller than $|z|$ and such that $zw_2 = zw_1w'_1\ldots w'_{n-1}w'_n$. Recording all relations that arise as $w_2 (w_1 w'_1 \ldots w'_n)\inv$ in the peak reduction lemma and combining them with the relations that hold in $\Aut_f(F)$ allow to deduce a presentation of $\Aut^0(F)$. Indeed, given a word representing the trivial element in $\Aut^0(F)$, one can apply the peak reduction lemma until no three consecutive initial subwords form a peak, i.e.\ until all initial subwords are of minimal complexity. This in fact implies that all letters are elements of $\Aut_f(F)$, so that the resulting word is a consequence of the relations in $\Aut_f(F)$. Some more work is needed to extract a finite presentation of $\Aut^0(F)$ under the hypotheses that each $F_j$ and each $\Aut(F_j)$ is finitely presented.
		
		In order to state the next theorem, we define $\Wc_\Fb := \Aut_i(F) / \Inn(F)$. As commented above, this notation is consistent with the definition of $\Wc_\Fb$ in Section~\ref{sec:structure_out}. We adapt the peak reduction method to recover a presentation of $\Aut_i(F)$, yielding the following analogue of Fouxe-Rabinovitch's theorem:
		\begin{theorem}\label{theorem:F-R}
			$\Aut_i(F)$ and $\Wc_\Fb$ are finitely presented provided that for each $j \in J$ the group $F_j$ is finitely presented and $\Centr(F_j)$ is finitely generated.
		\end{theorem}
		\begin{proof}
			Since each $F_j$ is finitely generated by assumption, so are $F$ and $\Inn(F)$. Since $\Wc_\Fb = \Aut_i(F) / \Inn(F)$ it is sufficient to show that $\Aut_i(F)$ is finitely presented in order to prove the theorem. 
			
			Set $\Aut_{if}(F) := \Aut_i(F) \cap \Aut_f(F)$. Clearly \[\Aut_{if}(F) \cong \prod_{j \in J} \Inn(F_j) \times \prod_{s \in S} \Aut(\Zb).\] Therefore the hypotheses that $F_j$ is finitely presented and that $\Centr(F_j)$ is finitely generated imply that $\Inn(F_j)$ is finitely presented, so that $\Aut_{if}(F)$ is finitely presented.
			
			It is easy to show that for any $x \in \Aut_f(F)$ and any $w \in \Wh$ we have $x w x\inv \in \Wh$. Therefore any element $z \in \Aut_\Fc(F)$ can be written as a product \[z= x w_1 w_2 \ldots w_n\] where $x \in \Aut_f(F)$ and each $w_i \in \Wh$. If moreover $z \in \Aut_i(F)$ then $x$ in the expression above must be in $\Aut_{if}(F)$ as $\Wh \subset \Aut_i(F)$. This shows that $\Aut_i(F)$ is generated by $\Omega' := \Wh \cup \Aut_{if}(F)$.
			
			Inspection of the relations in \cite{Gilb} reveals that peak reduction only involves elements of $\Omega'$, so that a presentation for $\Aut_i(F) = \la \Omega' \mid \Rc \ra$ is obtained by choosing $\Rc$ to be the union of the relations that allow peak reduction with all relations that hold in $\Aut_{if}(F)$. Finally one extracts a finite presentation of $\Aut_i(F)$ using finite presentations of $\Inn(F_j)$ and $F_j$ in a similar fashion as is done for $\Aut^0(F)$.
		\end{proof}
		
		To prove Theorem~\ref{theorem:out_fp}, we need to introduce a relative version of it.
		\begin{theorem}[Relative version of Theorem~\ref{theorem:out_fp}]
			Let $(\Ab,G,\Hc)$ be a triple as in Section~\ref{sec:structure_out} satisfying condition \ref{C1}.
			Suppose the following holds:
				\begin{enumerate}
					\item For each vertex $v$ of $T_A$ and each group $H$ in $\Hc$ contained in $G_v$, the normalizer $\Norm_{G_v} H$ of $H$ in $G_v$ is finitely presented, and its center $\Centr(\Norm_{G_v} H)$ is finitely generated. \label{rel_norm_fp_center_fg}
					\item For each vertex $v$ of $T_A$, the group $\Out_{\Hc_v}(G_v)$ of automorphisms relative to the family $\Hc_v$ of elements of $\Hc$ contained in $G_v$ is finitely presented. \label{rel_vertex_gp_fp}
				\end{enumerate}
			Then $\Out_\Hc(G)$ is finitely presented. 
		\end{theorem}
		\begin{proof}
			We proceed by induction on the number of edges of $\Ab$. If $|EA|=0$ then $\Ab$ is a single vertex, and the statement is trivially true by condition \eqref{vertex_gp_fp}. Suppose $|EA|>0$. Let $\Bb$ be as in Section~\ref{sec:structure_out}. Recall that for every vertex $v \in VB$ we have that $\Ab(v)$ has fewer edges than $\Ab$. Moreover $(\Ab(v),G_v,\Hc_v)$ is a triple as in Section~\ref{sec:structure_out} satisfying condition \ref{C1} and conditions~\eqref{rel_norm_fp_center_fg} and~\eqref{rel_vertex_gp_fp} hold for $\Ac(v)$ since they hold for $\Ac$. Hence inductions applies and $\Out_{\Hc_v}(G_v)$ is finitely presented for each $v \in V$. 
			
			Lemma~\ref{lemma:norm} and condition \eqref{rel_norm_fp_center_fg} imply that $N$ is finitely presented and that $\Centr(N_j)$ is finitely generated for each $j \in J$. The centralizer $\Centr_G(G_0)$ is finitely presented since it is a finite index subgroup of $N=\Norm_G(G_0)$. Observe that subgroups of finitely generated abelian groups are still finitely generated abelian, and that finitely generated abelian groups are finitely presented. Therefore $\Centr(N_j)$ is finitely presented. Moreover $\Centr(G_v) \subset \Centr(N_k)$ whenever $N_k \subset G_v$, so that $\Centr(G_v)$ is finitely generated for each $v \in V$. If $\Nb$ is a mapping torus (e.g. $|S|=1$ and $J=\emptyset$) then $\Centr(N)$ is virtually infinite cyclic, and $\Centr(N)$ is finite otherwise. In particular, $\Centr(N)$ is finitely generated. Putting all of this together with the exact sequence in the remark following Lemma~\ref{lemma:ker_sigma}, we get that $\ker \sigma^*$ is finitely presented.
			
			Again by Lemma~\ref{lemma:norm} and condition~\eqref{rel_norm_fp_center_fg}, $N_j$ is finitely presented and $\Centr(N_j)$ is finitely generated for each $j \in J$. Since $G_0$ is finite, $F_j = N_j/G_0$ is also finitely presented. We show that $\Centr(F_j) = \Centr(N_j/G_0)$ is finitely generated. Let $K_j$ be the preimage in $N_j$ of $\Centr(N_j/G_0)$. Recall from the proof of Corollary~\ref{corollary:ker_tau_G0_finite} that  $\Centr(N_j)$ has finite index in $\Centr_{N_j}(G_0) \cap K_j$. But $\Centr_{N_j}(G_0)$ has finite index in $N_j$ since $G_0$ is finite, so that $\Centr(N_j)$ is a finite index subgroup of $K_j$. Since $\Centr(N_j)$ is finitely generated, so is $K_j$. Finally $\Centr(F_j) = K_j / G_0$ and therefore $\Centr(N_j/G_0) = \Centr(F_j)$ is finitely generated. In conclusion each $F_j$ is finitely presented and each $\Centr(F_j)$ is finitely generated, so the group $\Wc_\Fb$ is finitely presented by Theorem~\ref{theorem:F-R}.
			
			Note that condition \eqref{rel_norm_fp_center_fg} implies confition \ref{C2}, so that the sequences in Proposition~\ref{proposition:main_tech} are virtually exact. By the second short sequence, $\Wc_\Bb$ is finitely presented. Using the first short sequence we get that $\Out^0(G)$ is finitely presented as well. Since we assumed condition \ref{C1} holds $\Out_\Hc(G)$ contains $\Out^0(G)$ as a finite index subgroup, hence $\Out_\Hc(G)$ is finitely presented.
		\end{proof}
		
		\begin{remark} In light of Lemma~\ref{lemma:edge_stab_canonical} Theorem~\ref{theorem:out_fp} is implied by its relative version.
		\end{remark}
		
	\section{Groups with infinite $\Out(G)$}
	\label{sec:finite_out}	
		
		Similarly to Section~\ref{sec:finite_pres}, we shall prove Theorem~\ref{theorem:finite_out} by induction on the number of edges of $\Ab$ using the following relative version:
		\begin{theorem}[Relative version of Theorem~\ref{theorem:finite_out}]
			Let $(\Ab,G,\Hc)$ be a triple as in Section~\ref{sec:structure_out} satisfying condition \ref{C1}. Then $\Out_\Hc(G)$ is infinite if and only if one of the following holds:
			\begin{enumerate}
				\item There is a vertex stabilizer $G_v$ of $T_A$ such that $\Out_{\Hc_v}(G_v)$ is infinite; \label{cond1}
				\item There is a splitting of $G$ as an amalgam $A \ast_C B$ over a finite group with $B \neq C$ such that the center of $A$ has infinite index in the centralizer of $C$ in $A$; \label{cond2}
				\item There is a splitting of $G$ as an HNN extension $A \ast_C$ over a finite group such that the centralizer of $\tilde C$ in $A$ is infinite, where $\tilde C$ is one of the two isomorphic copies of $C$ in $A$ given by the HNN extension. \label{cond3}
			\end{enumerate}
		\end{theorem}
		\begin{proof}
			One implication is straightforward. Suppose that $G$ satisfies condition \eqref{cond1}. 
			Let $\Hc_v^*$ be the set of edge stabilizers of $G$ contained in $G_v$. Let $\Out^*(G_v)$ be the finite index subgroup of $\Out_{\Hc_v}(G_v)$ acting on each element of $\Hc_v^*$ as a conjugation in $G_v$. Using Lemma~\ref{lemma:extension} one can extend any outer automorphism $\varphi$ of $\Out^*(G_v)$ to an outer automorphism $e(\varphi)$ of $G$. As $\Out(G_v)$ is infinite, then so is $G_v$. Hence $G_v$ is its own normalizer in $G$. Recall the definition of $\rho_v$ from Section~\ref{subsec:rho}. By definition, $e$ is such that $\rho_v(e(\varphi)) = \varphi$, so $\Out(G)$ must be infinite. If $G$ satisfies condition~\eqref{cond2} or~\eqref{cond3}, one easily constructs infinitely many Dehn twists using the given splitting. 
		
		Let us show by induction on the the number of edges of $\Ab$ that any accessible group satisfying the hypotheses of this theorem such that $\Out_\Hc(G)$ is infinite must satisfy one of the conditions \eqref{cond1}, \eqref{cond2} or  \eqref{cond3}. If $|EA|=0$, this is obvious as $\Ab$ consists of a single vertex. Suppose that $|EA|>0$. Thus $G$ is not one-ended and Section~\ref{sec:structure_out} applies to $\Ab$. We use the notation from Section~\ref{sec:structure_out} until the end of the proof. As we assumed condition \ref{C1} $\Out^0(G)$ has finite index in $\Out_\Hc(G)$ and so $\Out^0(G)$ is infinite. Combining Corollaries \ref{corollary:im_sigma_im_rho} and \ref{corollary:tau_virt_epi} there are four possibilities for $\Out^0(G)$ to be infinite (see Figure \ref{fig:CVS_groups}): either one of the vertex groups of $\Bb$ has infinite relative outer automorphism group; or $\ker \sigma^*$ is infinite; or $\ker \tau^*$ is infinite; or $\Wc_\Fb$ is infinite.
		
		For any $v \in V$ the graph of groups $\Ab(v)$ has strictly less edges than $\Ab$, so we know by induction that for any vertex $v \in VB$ with $\Out_{\Hc_v}(G_v)$ infinite, $G_v$ satisfies one of~\eqref{cond1}, \eqref{cond2} or \eqref{cond3}. As splittings of $G_v$ over finite groups can be extended to splittings of $G$, and since maximal elliptic subgroups of $G_v$ are maximal elliptic subgroups of $G$ acting on $T_A$, conditions \eqref{cond1}, \eqref{cond2} and \eqref{cond3} can be lifted to $G$, so we can assume that all vertex groups of $\Bb$ have finite relative automorphism group.
		
		Suppose the kernel of $\sigma^*$ is infinite. By Lemma~\ref{lemma:ker_sigma} only two cases might arise. Either there are $k \neq k' \in K$ such that $v_k=v_{k'} $ and such that $\Centr_G(N_k)$ is infinite. Thus collapsing all edges except $e_k$ in $\Bb$ yields an HNN-extension of $G$ over $G_0$ with $\Centr_G(G_0)$ infinite and such that no nontrivial power of the stable letter normalizes $G_0$, so that $G$ satisfies \eqref{cond3}. Or there is some $k$ such that $\Centr(G_v)$ has infinite index in $\Centr_G(N_k)$, yielding either an amalgam satisfying \eqref{cond2} or an HNN extension satisfying \eqref{cond3}.
		
		Suppose that the kernel of $\tau^*$ is infinite. Since $\Centr(G_0)$ is finite, Lemma~\ref{lemma:ker_tau} implies that there is some $j\in J$ such that $\Centr(N_j)$ has infinite index in $\Centr_{K_j}(G_0)$. Note that $\Centr(G_j) \subset \Centr(N_j)$ and that $\Centr_{K_j}(G_0) \subset \Centr_{G_j}(G_0)$, so there is some $j\in J$ such that $\Centr(G_j)$ has infinite index in $\Centr_{G_j}(G_0)$. Therefore, collapsing all edges of $\Bb$ except $e_j$ yields a splitting of $G$ satisfying \eqref{cond2}.
		
		Suppose now that $\Wc_\Fb$ is infinite. If the valence of $b_0$ is at least 3 in $\Fb$, either there is one non-loop edge and collapsing all other edges in $\Bb$ produces an amalgam satisfying \eqref{cond2}, or there are only loop edges and collapsing all but one in $\Bb$ yields an HNN extension satisfying \eqref{cond3}. If $|J|=2$ and $S$ is empty, then $F$ is a free product $F_j * F_{j'}$ and $\Wc_\Fb \cong \Inn(F_j) \times \Inn(F_{j'})$. Without loss of generality, the center of $F_j$ has infinite index in $F_j$, hence so does the center of $N_j$ in $N_j$ and $\Centr(G_j)$ in $\Norm_{G_j}(G_0)$. Hence, we have the desired amalgam. If either $S$ or $J$ consist of a single element and the other is empty, $\Wc_\Fb$ is finite, so we have examined all cases.
	\end{proof}
	
	\begin{remark} 
			Theorem~\ref{theorem:finite_out} is implied by its relative version by Lemma~\ref{lemma:edge_stab_canonical}.
	\end{remark}
	
	\section{Hyperbolic groups}
	\label{sec:hyperbolic_gps}
	
		Let us recall some facts about word hyperbolic groups:
		\begin{lemma}\label{lemma:std_hyp_facts}
			Let $\Gamma$ be a hyperbolic group. The following holds:
			\begin{enumerate}
				\item $\Gamma$ is finitely presented. \label{hyp_gp_fp}
				\item If $G_0$ is a finite subgroup of $\Gamma$, then the normalizer of $G_0$ in $\Gamma$ is a quasiconvex subgroup of $\Gamma$. In particular, it is hyperbolic and thus finitely presented.\label{hyp_gp_norm_finite_qc}
				\item There are finitely many conjugacy classes of finite subgroups of $\Gamma$. \label{hyp_gp_finite_subgps}
				\item The center of $\Gamma$ is either finite or virtually infinite cyclic. \label{hyp_gp_center}
			\end{enumerate}
		\end{lemma}
		\begin{proof} \eqref{hyp_gp_fp} The first assertion is proved in {\cite[5.2.3]{CoDePa}}.
		
			\eqref{hyp_gp_norm_finite_qc} By {\cite[3.3]{Short}} the centralizer of an element in $\Gamma$ is quasiconvex. Hence the centralizer of any finite subgroup $G_0$ is quasiconvex, as it is the intersection of the centralizer of each element of $G_0$, and so $\Centr_\Gamma(G_0)$ is the intersection of finitely many quasiconvex subgroups of $\Gamma$. Since $\Centr_\Gamma(G_0)$ is a finite index subgroup of $\Norm_\Gamma(G_0)$ the latter is quasiconvex as well. (See {\cite[Chap. 10]{CoDePa}} for facts about quasiconvex subgroups of hyperbolic groups).
			
			\eqref{hyp_gp_finite_subgps} The third assertion follows from the main theorem of \cite{BogGer}.
			
			\eqref{hyp_gp_center} If there is some $g \in \Centr(\Gamma)$ of infinite order, then $\Centr(\Gamma) \subset \Centr_\Gamma(\langle g \rangle)$ is virtually infinite cyclic by {\cite[10.7.2]{CoDePa}}. If not, then $\Centr(\Gamma)$ is an abelian torsion group. Since there is a bound on the order of a finite subgroup of $\Gamma$ by assertion \eqref{hyp_gp_finite_subgps}, any abelian torsion subgroup of $\Gamma$ is finite.
		\end{proof}
		
		\begin{lemma} \label{lemma:hyperbolic_splitting}
			If $G = A \ast_C$ (resp. $G=A \ast_C B$) where $C$ is a finite group, then $G$ is hyperbolic if and only if $A$ is hyperbolic (resp. $A$ and $B$ are hyperbolic).
		\end{lemma}
		\begin{proof}
			Suppose we have an amalgam $G=A \ast_C B$. If $A$ and $B$ are hyperbolic and $C$ is finite it can be seen directly from the Cayley graph that $G$ is hyperbolic. It also follows from the (much stronger) combination theorem in \cite{BestFeighn}. Suppose now that $G$ is hyperbolic. It is not hard to show that both $A$ and $B$ are quasiconvex subgroups of $G$, so that they are hyperbolic. The case of an HNN extension is identical.
		\end{proof}
		\begin{theorem}
			Let $\Gamma$ be a hyperbolic group. Then $\Aut(\Gamma)$ is finitely presented.
		\end{theorem}
		\begin{proof}
			Since $\Gamma$ is finitely presented, so by Dunwoody's accessibility \cite{Dunwoody} it admits a decomposition $\Ab$ as a finite reduced graph of groups with vertex groups with at most one end and finite edge groups. Let $\Hc$ be the family of edge stabilizers of $T_A$.
			
			Let us show that $\Ab$ satisfies condition \ref{vertex_gp_fp} of Theorem~\ref{theorem:out_fp}. Let $\Gamma_v$ be a vertex stabilizer of $T_A$. Applying Lemma~\ref{lemma:hyperbolic_splitting} repeatedly yields that $\Gamma_v$ is a hyperbolic group. As $\Gamma_v$ has at most one end, Levitt showed {\cite[Theorem 5.1]{Levitt}} that $\Out(\Gamma_v)$ is virtually an extension of $\prod_{x \in X} \Out_\Sc(G_x)$ by a group which is virtually $\Zb^n$, where $\Out_\Sc(G_x)$ is the algebraic mapping class group of a hyperbolic 2-orbifold $\Sigma_x$, i.e.\ the group of outer automorphisms of $G_x=\pi_1(\Sigma_x)$ preserving the family $\Sc$ of subgroups corresponding to boundary components. Following the proof in the orientable case \cite{MacHar}, Fujiwara \cite{Fujiwara} proved in full generality that when $\Sigma_x$ is a hyperbolic 2-orbifold, the algebraic mapping class group $\Out_\Sc(G_x)$ is isomorphic to the geometric mapping class group $\Mod^\partial(\Sigma_x)$ preserving the boundary componentwise. Moreover, he explains how to embed $\Mod^\partial(\Sigma_x)$ as a finite index subgroup of the mapping class group of a surface $\Mod^\partial(\Sigma'_x)$.
			Mapping class groups of surfaces are finitely presented by \cite{HatThur,McCool}. Therefore $\Out(\Gamma_v)$ is finitely presented. Let $\Hc_v$ be the set of edge stabilizers of $T_A$ contained in $\Gamma_v$. Since any hyperbolic group has only finitely many conjugacy class of finite subgroups, $\Out_{\Hc_v}(\Gamma_v)$ has finite index in $\Out(\Gamma_v)$ and so $\Out_{\Hc_v}(\Gamma_v)$ is finitely presented.
			
			$\Ab$ also satisfies condition \ref{norm_fp_center_fg} of Theorem~\ref{theorem:out_fp}. Indeed, any vertex stabilizer $\Gamma_v$ is hyperbolic, and so the normalizer $N$ of any finite group in $\Gamma_v$ is hyperbolic and finitely presented. Moreover, as $N$ is hyperbolic, its center is either finite or virtually infinite cyclic. In particular, it is finitely generated.
			
			Thus $\Ab$ satisfies the hypotheses of Theorem~\ref{theorem:out_fp}, so $\Out(\Gamma)$ is finitely presented. Moreover we have that
			\[ \Out(\Gamma) \cong \frac{\Aut(\Gamma)}{\Gamma/\Centr(\Gamma)}\] hence $\Aut(\Gamma)$ is finitely presented as well.
		\end{proof}
		
		Combining Theorem~\ref{theorem:finite_out} with the characterization of one-ended hyperbolic groups with infinite outer automorphism group in {\cite[Theorem 1.4]{Levitt}}, we get the following
		\begin{theorem}\label{theorem:hyp_gp_infinite_out}
			Let $\Gamma$ be a hyperbolic group. Then $\Out(\Gamma)$ is infinite if and only if one of the following holds:
			\begin{enumerate}
				\item $\Gamma$ splits as an amalgam of groups with finite center over a virtually cyclic subgroup with infinite center;
				\item $\Gamma$ splits as an arbitrary HNN extension over a virtually cyclic group with infinite center;
				\item $\Gamma$ splits as an amalgam $A \ast_C B$ over a finite group with $B \neq C$ such that the center of $A$ has infinite index in the centralizer of $C$ in $A$;
				\item $\Gamma$ splits as an HNN extension $A \ast_C$ over a finite group such that the centralizer of $\tilde C$ in $A$ is infinite, where $\tilde C$ is one of the two isomorphic copies of $C$ in $A$ given by the HNN extension.
				\qedhere
			\end{enumerate}
		\end{theorem}
		
		\section*{Acknowledgements}
			
			The author would like to thank Richard Weidmann for suggesting the problem and for his support during the preparation of this paper. He also profited from several discussions with Vincent Guirardel and Gilbert Levitt.
			
			The author is grateful to Christopher Deninger and Linus Kramer for inviting him as a Marie Curie fellow at the \emph{Universit\"at M\"unster}, Germany, where part of this research was undertaken.
			
		\bibliographystyle{amsplain}
		\bibliography{Biblio}
\end{document}

%% file: Rose.pstex_t
\begin{picture}(0,0)%
\includegraphics{Rose.pstex}%
\end{picture}%
\setlength{\unitlength}{4144sp}%
\begingroup\makeatletter\ifx\SetFigFont\undefined%
\gdef\SetFigFont#1#2#3#4#5{%
  \reset@font\fontsize{#1}{#2pt}%
  \fontfamily{#3}\fontseries{#4}\fontshape{#5}%
  \selectfont}%
\fi\endgroup%
\begin{picture}(1890,1145)(436,-564)
\put(451,434){\makebox(0,0)[lb]{\smash{{\SetFigFont{10}{12.0}{\familydefault}{\mddefault}{\updefault}{\color[rgb]{0,0,0}$v_j=v_{j'}$}%
}}}}
\put(2251,299){\makebox(0,0)[lb]{\smash{{\SetFigFont{10}{12.0}{\familydefault}{\mddefault}{\updefault}{\color[rgb]{0,0,0}$e_s$}%
}}}}
\put(451,254){\makebox(0,0)[lb]{\smash{{\SetFigFont{10}{12.0}{\familydefault}{\mddefault}{\updefault}{\color[rgb]{0,0,0}$=v_i$}%
}}}}
\put(1531,-421){\makebox(0,0)[lb]{\smash{{\SetFigFont{10}{12.0}{\familydefault}{\mddefault}{\updefault}{\color[rgb]{0,0,0}$b_0$}%
}}}}
\put(1351,299){\makebox(0,0)[lb]{\smash{{\SetFigFont{10}{12.0}{\familydefault}{\mddefault}{\updefault}{\color[rgb]{0,0,0}$e_j$}%
}}}}
\put(1216, 74){\makebox(0,0)[lb]{\smash{{\SetFigFont{10}{12.0}{\familydefault}{\mddefault}{\updefault}{\color[rgb]{0,0,0}$e_{j'}$}%
}}}}
\put(901,-151){\makebox(0,0)[lb]{\smash{{\SetFigFont{10}{12.0}{\familydefault}{\mddefault}{\updefault}{\color[rgb]{0,0,0}$e_i$}%
}}}}
\end{picture}%

%% file: Normalizer.pstex_t
\begin{picture}(0,0)%
\includegraphics{Normalizer.pstex}%
\end{picture}%
\setlength{\unitlength}{4144sp}%
\begingroup\makeatletter\ifx\SetFigFont\undefined%
\gdef\SetFigFont#1#2#3#4#5{%
  \reset@font\fontsize{#1}{#2pt}%
  \fontfamily{#3}\fontseries{#4}\fontshape{#5}%
  \selectfont}%
\fi\endgroup%
\begin{picture}(2205,1280)(121,-609)
\put(1036,524){\makebox(0,0)[lb]{\smash{{\SetFigFont{10}{12.0}{\familydefault}{\mddefault}{\updefault}{\color[rgb]{0,0,0}$N_j$}%
}}}}
\put(631,389){\makebox(0,0)[lb]{\smash{{\SetFigFont{10}{12.0}{\familydefault}{\mddefault}{\updefault}{\color[rgb]{0,0,0}$N_{j'}$}%
}}}}
\put(136,119){\makebox(0,0)[lb]{\smash{{\SetFigFont{10}{12.0}{\familydefault}{\mddefault}{\updefault}{\color[rgb]{0,0,0}$N_i=G_0$}%
}}}}
\end{picture}%

%% file: bprime.pstex_t
\begin{picture}(0,0)%
\includegraphics{bprime.pstex}%
\end{picture}%
\setlength{\unitlength}{4144sp}%
\begingroup\makeatletter\ifx\SetFigFont\undefined%
\gdef\SetFigFont#1#2#3#4#5{%
  \reset@font\fontsize{#1}{#2pt}%
  \fontfamily{#3}\fontseries{#4}\fontshape{#5}%
  \selectfont}%
\fi\endgroup%
\begin{picture}(1020,1055)(661,-564)
\put(1666,-241){\makebox(0,0)[lb]{\smash{{\SetFigFont{10}{12.0}{\familydefault}{\mddefault}{\updefault}{\color[rgb]{0,0,0}$N$}%
}}}}
\put(1351,299){\makebox(0,0)[lb]{\smash{{\SetFigFont{10}{12.0}{\familydefault}{\mddefault}{\updefault}{\color[rgb]{0,0,0}$N_k$}%
}}}}
\put(676,344){\makebox(0,0)[lb]{\smash{{\SetFigFont{10}{12.0}{\familydefault}{\mddefault}{\updefault}{\color[rgb]{0,0,0}$G_k$}%
}}}}
\end{picture}%